\documentclass[11pt]{article} 
\usepackage[utf8]{inputenc} 

\usepackage[margin=1in]{geometry} 
\usepackage{graphicx} 

\usepackage{dsfont} 
\usepackage{fvextra}

\usepackage{booktabs} 
\usepackage{array} 
\usepackage{paralist} 
\usepackage{verbatim} 
\usepackage{mathrsfs}
\usepackage{amssymb}
\usepackage{amsthm}
\usepackage{amsmath,amsfonts,amssymb}
\usepackage{esint}
\usepackage{graphics}
\usepackage{enumerate}
\usepackage{mathtools}
\usepackage{xfrac}
\usepackage{subcaption}
\usepackage{stmaryrd}
 \usepackage{mathabx}

\usepackage[dvipsnames]{xcolor}
\usepackage[colorlinks=true, pdfstartview=FitV, linkcolor=blue, citecolor=blue, urlcolor=blue]{hyperref}
\usepackage[normalem]{ulem}

\usepackage{tikz}
\usetikzlibrary{calc}
\usepackage{pgf}
\usetikzlibrary{external}
\numberwithin{equation}{section}
\numberwithin{figure}{section}

\newtheorem{theorem}{Theorem}[section]

\newtheorem{proposition}[theorem]{Proposition}
\newtheorem{lemma}[theorem]{Lemma}
\theoremstyle{definition}

\newtheorem{remark}[theorem]{Remark}

\newcommand*{\supp}{\ensuremath{\mathrm{supp\,}}}

\newcommand*{\N}{\ensuremath{\mathbb{N}}}

\newcommand*{\Z}{\ensuremath{\mathbb{Z}}}

\newcommand*{\R}{\ensuremath{\mathbb{R}}}

\newcommand{\eps}{\varepsilon}

\renewcommand*{\tilde}{\widetilde}

\newcommand{\ep}{\eps}

\DeclareSymbolFont{boldoperators}{OT1}{cmr}{bx}{n}
\SetSymbolFont{boldoperators}{bold}{OT1}{cmr}{bx}{n}
\usepackage{accents}

\newcommand{\T}{\mathbb{T}}

\def\XXint#1#2#3{{\setbox0=\hbox{$#1{#2#3}{\int}$}
\vcenter{\hbox{$#2#3$}}\kern-.5\wd0}}

\let\originalleft\left
\let\originalright\right
\renewcommand{\left}{\mathopen{}\mathclose\bgroup\originalleft}
\renewcommand{\right}{\aftergroup\egroup\originalright}

\newcommand{\C}{\mathbb{C}}

\newcommand{\sol}{\mathcal{T}}
\newcommand{\E}{\mathbb{E}}

\renewcommand{\hat}{\widehat}

\makeatletter
\pgfmathdeclarefunction{erf}{1}{%
  \begingroup
    \pgfmathparse{#1 > 0 ? 1 : -1}%
    \edef\sign{\pgfmathresult}%
    \pgfmathparse{abs(#1)}%
    \edef\x{\pgfmathresult}%
    \pgfmathparse{1/(1+0.3275911*\x)}%
    \edef\t{\pgfmathresult}%
    \pgfmathparse{%
      1 - (((((1.061405429*\t -1.453152027)*\t) + 1.421413741)*\t 
      -0.284496736)*\t + 0.254829592)*\t*exp(-(\x*\x))}%
    \edef\y{\pgfmathresult}%
    \pgfmathparse{(\sign)*\y}%
    \pgfmath@smuggleone\pgfmathresult%
  \endgroup
}
\makeatother

\usepackage{titlesec}

\newcommand{\addperiod}[1]{#1.}
\titleformat{\section}
   {\centering\normalfont\Large}{\thesection.}{0.5em}{}
\titleformat*{\subsection}{\bfseries}
\titleformat{\subsubsection}[runin]
  {\normalfont\bfseries}
  {\thesubsubsection.}
  {0.5em}
  {\addperiod}
\titleformat*{\subsubsection}{\normalfont\itshape}
\titleformat*{\paragraph}{\bfseries}
\titleformat*{\subparagraph}{\large\bfseries}

\title{An AI-discovered smooth random fast dynamo on $\mathbb{T}^3$}

\author{
Keefer Rowan\thanks{Institute of Mathematics, \'Ecole Polytechnique F\'ed\'erale de Lausanne.
{\footnotesize \href{mailto:keefer.rowan@epfl.ch}{keefer.rowan@epfl.ch}.}
}
}
\date{\today}

\usepackage[nottoc,notlot,notlof]{tocbibind}

\begin{document}

\maketitle

\begin{abstract}
    We construct a random, time-dependent divergence-free velocity field on $\T^3$---refreshing iid on finite time blocks and obeying deterministic $C^\infty_{t,x}$ bounds---that exhibits fast dynamo behavior. That is, for every fixed, sufficiently small resistivity, the almost sure exponential growth rate of the magnetic field solving the associated linear resistive induction equation is at least $1/2$; the exceptional null set may depend on the resistivity. We in fact get a time-uniform lower bound---with a random prefactor obeying a uniform-in-$\kappa$ inverse moment bound. The argument relies on a particular algebraic structure in Fourier space of the induction equation solution operator that allows us to propagate expected growth of the logarithmic size of three specially chosen Fourier modes. This allows us to reduce to a simple recursion, avoiding the complicated infinite-dimensional dynamics typical to the dynamo problem. The central proof idea was generated autonomously by ChatGPT 5.6 Sol Ultra; the manuscript was written (and verified) by the author.
\end{abstract}

\section{Introduction}
Let
    \begin{align*}
    u_0(x,y,z) &= \mathrm{e}_z \sin\big(2\pi(x-y)\big)\\
    u_1(x,y,z) &= \mathrm{e}_x \sin\big(2\pi(y-z)\big)\\
    u_2(x,y,z) &= \mathrm{e}_y\sin\big(2\pi(z-x)\big).
\end{align*}

For $j \in \Z$, let $\theta_j$ be iid, uniformly distributed $\T^3$-valued random variables. Let $\eta : \R \to [0,\infty)$ be a smooth bump function: $\supp \eta \subseteq [0,1]$ and $\int \eta(s)\,ds = 1$.
Define the random velocity field
\[U(t,x,y,z) := \sum_{j \in \Z} \eta(t-j) u_{j \mathrm{\,mod\,} 3}(x+\theta_{j,x}, y+\theta_{j,y}, z + \theta_{j,z}).\]

We note that $\nabla \cdot U =0$ and for all $\ell, k \in \N$, there exists $C(k,\ell)>0$ such that we have the deterministic, pathwise bound:
\[\sup_{t \in \R, (x,y,z) \in \T^3} \big|\partial_t^\ell \nabla^k U(t,x,y,z)| \leq C(k,\ell).\]
For $\kappa \geq 0$, we let $B^\kappa : [0,\infty) \times \T^3 \to\R^3$ be the (random) unique solution to
\begin{equation}
    \label{eq:B kappa def}
    \begin{cases}
    \partial_t B^\kappa - \kappa \Delta B^\kappa + U \cdot \nabla B^\kappa - B^\kappa \cdot \nabla U=0\\
    B^\kappa(0,x,y,z) = -2\sqrt{2} \mathrm{e}_y \cos(2\pi x).
    \end{cases}
\end{equation}
Note that $\nabla \cdot B^\kappa_0 = \nabla \cdot B^\kappa_t =0.$ The main result is that $U$ induces fast dynamo growth in the magnetic field $B^\kappa$ solving~\eqref{eq:B kappa def}.

\begin{theorem}
    \label{thm:main}
     For all $\kappa \in \big[0,\frac{1}{4\pi^2} \log\frac{\pi}{e}\big]$ there exists an almost surely finite random constant $L^\kappa \geq 1$ such that for all $t \geq 0$, 
    \[\|B^\kappa_t\|_{L^2_x} \geq \tfrac{1}{L^\kappa} e^{t/2}.\]
    Further, we have the moment estimate that for some $p,C >0$, taken uniformly for $\kappa \in \big[0,\frac{1}{4\pi^2} \log\frac{\pi}{e}\big]$,
    \[\E (L^\kappa)^p \leq C.\]
\end{theorem}

\begin{remark}
\label{rem:no fast dynamo}
    We note that the random constant is almost surely finite and obeys moment bounds uniformly in $\kappa \in \big[0,\frac{1}{4\pi^2} \log\frac{\pi}{e}\big]$. However, $L^\kappa$ \textit{depends on $\kappa$}. Hence this proof does not immediately produce a single deterministic field that causes exponential growth for all $\kappa \in \big[0,\frac{1}{4\pi^2} \log\frac{\pi}{e}\big]$, as that would require intersecting uncountably many full probability events. One does immediately get from Fubini's theorem that there is asymptotic exponential growth for almost every $\kappa \in \big[0,\frac{1}{4\pi^2} \log\frac{\pi}{e}\big]$ (of course with a $\kappa$-dependent prefactor). An analogous phenomenon holds in uniform-in-diffusivity exponential mixing~\cite{bedrossian_almost-sure_2021,cooperman_exponentially_2025}.
\end{remark}

\subsection{Dynamos}

Kinematic dynamo theory refers to the study of exponentially growing solutions to the magnetic induction equation
\begin{equation}
\label{eq:induction}
\partial_t b^\kappa - \kappa \Delta b^\kappa + u \cdot \nabla b^\kappa - b^\kappa \cdot \nabla u=0,
\end{equation}
where $u : [0,\infty) \times \T^3 \to \R^3$, $\nabla \cdot u =0$ is an incompressible advecting flow and $\nabla \cdot b^\kappa=0$---which is propagated for free if it holds for $b^\kappa_0$---as required for a magnetic field under Maxwell's equations. The magnetic induction equation models the evolution of a magnetic field in a conducting fluid---neglecting the influence of the magnetic field on the fluid. The growth of the magnetic field under the induction equation is a proposed explanation for the persistent magnetic field of astrophysical bodies, such as the Sun and the Earth, going back to~\cite{larmor_how_1919}.

The dynamo problem refers to building velocity fields for which one can prove exponential growth of the magnetic field. This problem has seen substantial treatment in the applied literature---see~\cite{childress_stretch_1995} for a well-written overview. Of particular interest is the construction of a fast dynamo, for which the exponential growth rate is bounded below by a positive constant, independent of the resistivity $\kappa \in [0,\kappa_0]$ for some $\kappa_0>0$. The rigorous construction of a fast dynamo is stated as one of Arnold's problems~\cite[Problem 1994-28]{arnold_arnolds_2004}. 

To make the problem precise, letting $b^\kappa$ be a solution to~\eqref{eq:induction} with initial data $b_0$, we define the dynamo rate for the velocity field $u$ and the resistivity $\kappa \geq 0$ by
\[\gamma(u,\kappa) := \sup_{b_0 \in L^2(\T^3), \nabla \cdot b_0 =0} \liminf_{t \to \infty} \frac{1}{t} \log \|b^\kappa_t\|_{L^2_x}.\]
Usually, a flow $u$ is said to be a fast dynamo if $\gamma(u,\kappa) \geq \gamma_0 > 0$ for all $\kappa \in [0,\kappa_0]$. As noted in Remark~\ref{rem:no fast dynamo} above, the proof does not almost surely give flows that are deterministically fast dynamos, but rather gives the natural random analog, where for a random field $U$, for all $\kappa \in [0,\kappa_0]$ we have the almost sure inequality (with a $\kappa$-dependent null event), $\gamma(U,\kappa) \geq \gamma_0 > 0$. We believe it is appropriate to call such random flows (random) fast dynamos.

Until recently, rigorous work on dynamos largely did not address the positive problem of constructing dynamos on (physically natural) flat spaces. Dynamos, however, have been constructed on negatively curved spaces~\cite{chicone_geodesic_1997} and a variety of necessary conditions for dynamo action have been introduced~\cite{cowling_magnetic_1933,zeldovich_magnetic_1956,friedlander_dynamo_1991,klapper_rigorous_1995}. In the last few years, there has been major progress on the flat space problem, much of it using sophisticated spectral arguments, starting with \cite{zelati_alpha-unstable_2025}, which builds a fast dynamo on $\R^3$. As physical dynamos have their magnetic field growth visible in a compact region of space, much attention has been focused on the simplest compact case of $\T^3$, which is the setting for the remaining results we discuss. \cite{sorella_limsup_2025} builds a $\limsup$ fast dynamo---where the $\liminf_t$ in the dynamo rate definition is replaced with a $\limsup_t$---and \cite{zelati_fast_2026,niebel_autonomous_2026} build fast dynamos. These constructions are all at the Lipschitz spatial regularity, which is a natural critical regularity for the problem. Nevertheless, one expects dynamo action to also occur for smooth flows. \cite{navarro-fernandez_spectral_2025} builds a smooth \textit{slow} dynamo, where $\gamma(u,\kappa)>0$ for $\kappa \in (0,\kappa_0)$ but $\gamma(u,\kappa) \to 0$ as $\kappa \to 0$ (technically they work on a variety of manifolds which are not $\T^3$, though their results cover the compact manifold $\{|(x,y)| \leq 1\} \times \T$). \cite{rowan_subsequentially_2025} builds a smooth \textit{subsequentially} (and $\limsup$ in the sense of \cite{sorella_limsup_2025}) fast dynamo, where $\gamma(u,\kappa)$ (defined with a $\limsup_t$) stays bounded away from $0$ only on a sequence of resistivities $\kappa_j \to 0.$ There is also the theory of ideal dynamos (the $\kappa=0$ case)~\cite{navarro-fernandez_ideal_2026}, for which smooth examples are known~\cite{coti_zelati_three-dimensional_2026}, though the phenomenology there is fairly different. The result of this work provides the first smooth random fast dynamo.

All of these constructions are fairly bespoke, being built specifically in order to show the desired dynamo behavior. One could hope that a general theory is attainable for \textit{random flows}, where suitable random velocity fields with an appropriate non-degeneracy condition are fast dynamos---analogous to the striking results in the problem of uniform-in-diffusivity exponential mixing~\cite{bedrossian_almost-sure_2021,cooperman_exponentially_2025}, which shares many similarities to the dynamo problem. This is very physically appealing in providing a generic origin for dynamo growth. This work does not provide such a theory, though perhaps goes somewhat in the desired direction.

The reason one may expect randomness to help is that the noise can prevent cancellations that cause a possible route to exponential growth to be perfectly erased by an alternative mechanism. Showing that randomness can perform this role in a general setting is a quite difficult (and very interesting) open problem. However, this result gives a (very) simple example of such a phenomenon based on a special algebraic structure.

\subsection{Discussion of the proof}

For any fixed $\kappa$ and deterministic velocity field $u: \T^3 \to \R^3$, let $\mathcal{T}$ denote the unit time solution operator to~\eqref{eq:induction} and let $\mathcal{T}(\theta)$ denote the unit time solution operator to~\eqref{eq:induction} with velocity field $u(\cdot + \theta)$, where $\theta \in \T^3$. A straightforward computation then shows that
\[\hat \sol(\theta; k,j) :=\int e^{-2\pi i k \cdot r} \mathcal{T}(\theta) e^{2\pi i j \cdot r}\,dr =  e^{2\pi i (k-j) \cdot \theta} \int e^{-2\pi i k \cdot r} \mathcal{T} e^{2\pi i j \cdot r}\,dr =: e^{2\pi i (k-j) \cdot \theta} \hat \sol(k,j). \]
Thus for arbitrary fields $u$, random translations generate a particular Fourier structure in the translation parameter, which can be easily exploited (for suitable choice of $u$) to get both exponential growth in expectation and exponential growth along \textit{some} noise trajectories (though perhaps a measure zero set of them). This is the central idea of~\cite{rowan_subsequentially_2025}.

However, generally, even if we know \textit{some} Fourier coefficients are large---hence we know that we get growth for \textit{some} $\theta$---it is highly nonobvious---and potentially generally false---that we should get growth for \textit{typical} $\theta$. E.g.\ an approximate $\delta$ mass can have large Fourier coefficients but, with arbitrarily high probability, be quite small.

The central innovation of this manuscript---discovered by ChatGPT 5.6 Sol Ultra, see further discussion of AI use in Section~\ref{s:AI} below---is to use that for well-chosen velocity fields, we can make $\hat \sol(k,j)$ ``tridiagonal'' on certain rows; see Proposition~\ref{prop:one pulse}. We can then use a special fact---Lemma~\ref{lem:Jensen applied} coming from Jensen's formula\footnote{A complex analysis fact; not the Jensen inequality.}---that lower bounds the average $\log$ magnitude of finite polynomials in $e^{2\pi i \varphi}$ by the logarithm of their extremal coefficients:
\[\int_0^1 \log |ae^{-2\pi i \varphi} + b + ce^{2\pi i \varphi}|\,d\varphi \geq \log\big(\max(|a|,|c|)\big).\]
By extremal coefficients, we mean the coefficients on the lowest and highest degree terms---notably the intermediate-degree terms do not work without incurring a degree-dependent loss, hence requiring the special tridiagonal structure as opposed to the more typical case of an infinite-order Fourier series in $e^{2\pi i \varphi}$.

These two facts together allow us to iteratively propagate linear growth in the expected logarithmic size of certain Fourier coefficients. A further martingale argument---again exploiting the special structure of the polynomials in $e^{2\pi i \varphi}$---then enforces that the fluctuations around the expected growth are asymptotically non-dominant, leading to the all-time lower bound statement of Theorem~\ref{thm:main}.

\section{Proof of the main result}

Throughout, for a function $f : \T^3 \to \C$, we denote the Fourier coefficient $\hat f(k)$ for $k \in \Z^3$:
\[\hat f(k) := \int e^{-2\pi i k \cdot r} f(r)\,dr.\]

\begin{proposition}
\label{prop:one pulse}
    Let $\varphi$ be a uniform $\T^1$-valued random variable and let $v : [0,1] \times \T^3 \to \R^3$ be the random velocity field given by
    \[v(t,x,y,z) = \eta(t) u_0(x+\varphi,y,z) = \mathrm{e}_z\eta(t) \sin\big(2\pi(x-y + \varphi)\big).\]
    For $\kappa \geq 0$, let $b^\kappa : [0,1] \times \T^3 \to \R^3$ be the (random) unique solution to
    \begin{equation}
    \partial_t b^\kappa - \kappa \Delta b^\kappa + v\cdot \nabla b^\kappa - b^\kappa \cdot \nabla v=0,
\end{equation}
    with deterministic initial data $b^\kappa_0 = b_0 \in L^2(\T^3)$. Then
    \begin{align*}\tfrac{1}{\sqrt{2}}(\mathrm{e}_y - \mathrm{e}_z) \cdot \hat b_1^\kappa(0,1,0) &= A^{\kappa,-1} \tfrac{1}{\sqrt{2}}(\mathrm{e}_x - \mathrm{e}_y) \cdot\hat b_{0}(1,0,0) e^{-2\pi i \varphi}+ A^{\kappa,0} \tfrac{1}{\sqrt{2}} (\mathrm{e}_y - \mathrm{e}_z) \cdot \hat b_{0}(0,1,0)
    \\&\qquad + A^{\kappa,1}\tfrac{1}{\sqrt{2}}(\mathrm{e}_x - \mathrm{e}_y) \cdot \hat b_{0}(-1,2,0)e^{2\pi i \varphi},
    \end{align*}
    where
       \begin{align*}
        A^{\kappa,1} &:= -\pi \int_0^1 \eta(t) e^{-4\pi^2 \kappa(1-t)} e^{-20 \pi^2 \kappa t}\,dt,\\
        A^{\kappa,0} &:= e^{-4\pi^2 \kappa},\\
        A^{\kappa,-1} &:= -\pi e^{-4\pi^2 \kappa}.
    \end{align*}
\end{proposition}

\begin{proof}
    Write $b_t^\kappa = \sum_{j \in \Z} e^{2\pi i j z} \beta_t^j(x,y)$ with $\beta^j_t : \T^2 \to \C^3$. We note that, as $v$ is independent of the $z$ coordinate, the $\beta^j$ solve closed equations. We are interested only in $\beta := \beta^0$, for which the advective term trivializes, as there is no $z$ dependence and $v$ points purely in the $z$ direction. That is $\beta$ solves
    \[\partial_t \beta = \kappa \Delta \beta +\beta \cdot \nabla v = \kappa \Delta \beta + 2\pi \eta(t) \mathrm{e}_z \Big( (\beta_x - \beta_y) \cos\big(2\pi (x-y + \varphi)\big)\Big).\]
    Thus the vector stretching term only affects the $z$ coordinate of $\beta$, so
    \[\beta_{t,x} = e^{\kappa t \Delta} \beta_{0,x} \quad \text{and} \quad \beta_{t,y} = e^{\kappa t \Delta} \beta_{0,y}.\]
    Then
    \[\partial_t \beta_{t,z} =  \kappa \Delta \beta_{t,z} + 2\pi \eta(t) \Big( e^{\kappa t \Delta} (\beta_{0,x}- \beta_{0,y}) \cos\big(2\pi (x-y + \varphi)\big)\Big).\]
    Thus
    \begin{align*}&\partial_t \hat \beta_{t,z}(0,1) = - 4 \pi^2 \kappa  \hat \beta_{t,z}(0,1) 
    \\&\qquad+ \pi \eta(t) \Big(e^{-20 \pi^2 \kappa t} (\hat\beta_{0,x}(-1,2)- \hat\beta_{0,y}(-1,2))e^{2\pi i \varphi}  +  e^{-4 \pi^2 \kappa t} (\hat\beta_{0,x}(1,0)- \hat\beta_{0,y}(1,0))e^{-2\pi i \varphi}\Big),
    \end{align*}
    where we compute
\begin{align*}&\int e^{-2\pi i y}e^{\kappa t \Delta} (\beta_{0,x}- \beta_{0,y}) \big(e^{2\pi i (x-y + \varphi)} + e^{-2\pi i(x-y + \varphi)}\big) \,dxdy
\\&\qquad= \int (\beta_{0,x}- \beta_{0,y}) e^{\kappa t \Delta} \big(e^{2\pi i (x-2y + \varphi)} + e^{-2\pi i(x + \varphi)}\big) \,dxdy
\\&\qquad =e^{-20 \pi^2 \kappa t} (\hat\beta_{0,x}(-1,2)- \hat\beta_{0,y}(-1,2))e^{2\pi i \varphi}  +  e^{-4 \pi^2 \kappa t} (\hat\beta_{0,x}(1,0)- \hat\beta_{0,y}(1,0))e^{-2\pi i \varphi} .
\end{align*}
    Thus
    \begin{align*}\hat \beta_{1,z}(0,1) &= e^{-4 \pi^2 \kappa} \hat \beta_{0,z}(0,1) + \pi \Big(\int_0^1 \eta(t) e^{-4\pi^2 \kappa(1-t)} e^{-20 \pi^2 \kappa t}\,dt\Big) (\hat\beta_{0,x}(-1,2)- \hat\beta_{0,y}(-1,2))e^{2\pi i \varphi} 
    \\&\qquad + \pi \Big(\int_0^1 \eta(t) e^{-4\pi^2 \kappa(1-t)} e^{-4 \pi^2 \kappa t}\,dt\Big)(\hat\beta_{0,x}(1,0)- \hat\beta_{0,y}(1,0))e^{-2\pi i \varphi}.
    \end{align*}
    Thus, we finally have that
    \begin{align*}
        (\mathrm{e}_y - \mathrm{e}_z) \cdot\hat b_1^\kappa(0,1,0) &= \hat \beta_{1,y} - \hat \beta_{1,z}
        \\&= A^{\kappa,0}(\hat \beta_{0,y}(0,1) - \hat \beta_{0,z}(0,1)) + A^{\kappa,1} (\hat\beta_{0,x}(-1,2)- \hat\beta_{0,y}(-1,2))e^{2\pi i \varphi} 
    \\&\qquad + A^{\kappa,-1}(\hat\beta_{0,x}(1,0)- \hat\beta_{0,y}(1,0))e^{-2\pi i \varphi},
    \end{align*}
    with $A^{\kappa,j}$ as in the proposition statement. Recalling that $\hat \beta_t(j,\ell) = \hat b_t^\kappa(j,\ell,0)$ and dividing by $\sqrt{2}$, we conclude. 
\end{proof}

We will use the following fact, which is just Jensen's formula~\cite[Page 208]{ahlfors_complex_1978} specialized to complex linear factors.
\begin{lemma}
\label{lem:Jensen}
    For all $r \in \C$,
    \[\int_0^1 \log \big|e^{2\pi i \varphi} - r\big|\,d\varphi = \log\big(\max(1,|r|)\big).\]
\end{lemma}

\begin{lemma}
\label{lem:Jensen applied}
    Let $\varphi$ be a uniform $\T^1$-valued random variable and $a,b,c \in \C$ and $|a| + |b| + |c| >0$. Let $X$ be the $\R$-valued random variable given by $X := \log \big|a e^{- 2\pi i \varphi} + b + c e^{2\pi i \varphi}\big|$. Then we have that
    \begin{equation}\label{eq:mean growth}
    \E X \geq \log\big(\max(|a|,|c|)\big).
    \end{equation}
    Further, for all $p \in (0,1/2)$, there exists $C(p)>0$ (not depending on $a,b,c$) such that 
   \begin{equation} \E e^{p |X - \E X|} \leq C(p). 
    \label{eq:fluctuation bound}
    \end{equation}
\end{lemma}

\begin{proof}
    Sending $\varphi \mapsto -\varphi$ if necessary, we can suppose without loss of generality that $|a| \geq |c|$. If $|a| =0$, then the result holds trivially. Otherwise, letting $w \in \C$ be given by $w = e^{-2\pi i \varphi}$, we note that 
    \[X = \log \big|a w^2 + bw + c| = \log |a| + \log |w^2 + a^{-1} b w + a^{-1} c|.\]
    Thus the result with $a=1$ and $|c| \leq 1$ implies the general result; we thus take this hypothesis.

    Then we can factor the polynomial, giving for some $r_1, r_2 \in \C$,
    \[X = \log |w^2 +  b w + c| = \log |w- r_1| + \log |w-r_2|.\]
    Then by Lemma~\ref{lem:Jensen},
    \[\E X = \log \big(\max(1,|r_1|)\big) +\log \big(\max(1,|r_2|)\big).\]
    Thus $\E X \geq 0$ allows us to conclude~\eqref{eq:mean growth}. For~\eqref{eq:fluctuation bound}, we note that
    \[X - \E X = \log \frac{|w-r_1|}{\max(1,|r_1|)} + \log \frac{|w-r_2|}{\max(1,|r_2|)}.\]
    Thus $X - \E X\leq 2 \log 2$. Then for any $p >0$,
    \[\E e^{p|X - \E X|} \leq \E e^{p(X - \E X)} + \E e^{p(\E X - X)} \leq 4^p+ \E e^{p(\E X - X)}.\]
    We also have that
    \begin{align*}
        \E e^{p(\E X - X)} &=\int_0^1 \frac{\max(1,|r_1|)^p  \max(1,|r_2|)^p}{|e^{2\pi i \varphi} - r_1|^p |e^{2\pi i \varphi} - r_2|^p}\,d\varphi
        \\&\leq \Big(\int_0^1 \frac{\max(1,|r_1|)^{2p}}{|e^{2\pi i \varphi} - r_1|^{2p}}\,d\varphi\Big)^{1/2}\Big(\int_0^1 \frac{\max(1,|r_2|)^{2p}}{|e^{2\pi i \varphi} - r_2|^{2p}}\,d\varphi\Big)^{1/2}.
    \end{align*}
    Then if $|r_1| \geq 1$, we have that 
   \[ \int_0^1 \frac{\max(1,|r_1|)^{2p}}{|e^{2\pi i \varphi} - r_1|^{2p}}\,d\varphi = \int_0^1 |r_1^{-1} e^{2\pi i \varphi} - 1|^{-2p}\,d\varphi,\]
   while if $|r_1|\leq 1$, we have that
   \[\int_0^1 \frac{\max(1,|r_1|)^{2p}}{|e^{2\pi i \varphi} - r_1|^{2p}}\,d\varphi =\int_0^1 | e^{2\pi i \varphi} - r_1|^{-2p}\,d\varphi= \int_0^1 |1-r_1 e^{-2\pi i \varphi}|^{-2p}\,d\varphi =  \int_0^1 |r_1 e^{2\pi i \varphi}-1|^{-2p}\,d\varphi.\]
   Thus 
   \[\E e^{p(\E X - X)} \leq \sup_{|d| \leq 1} \int_0^1 |d e^{2\pi i\varphi} -1|^{-2p}\,d\varphi \leq C(p),\]
    where we have a finite constant $C(p)>0$ for any $p \in (0,1/2)$, using that $x^{-q}$ is integrable on $[0,1]$ whenever $q <1$.

    Combining the above estimates, we conclude~\eqref{eq:fluctuation bound}.
\end{proof}

\begin{lemma}
\label{lem:martingale term subcritical}
    Let $(X_j)_{j \in \N}$ be a sequence of integrable $\R$-valued random variables adapted to a filtration $(\mathcal{F}_j)_{j \in \N}$. Suppose that for some $h>0$, for all $n \in \N$, 
    \[\E[X_{n+1} \mid \mathcal{F}_n] \geq h + X_n,\]
    and that for some $p>0$, there exists $K>0$ such that, almost surely,
    \[\E\bigg[ \exp\Big(p \big| X_{n+1} - \E[X_{n+1} \mid \mathcal{F}_n] \big|\Big)\, \bigg|\, \mathcal{F}_n\bigg] \leq K.\]
    Then there exists an almost surely finite random variable $R \geq 0$ such that for all $n \in \N$, we have
    \[X_n \geq X_0+ \frac{h}{2} n - R\]
    and $R$ obeys the exponential moment estimate for some $q(p,K,h) >0$ and $C(p,K,h)>0$,
    \[\E e^{qR} \leq C.\]
\end{lemma}

\begin{proof}
    Let
    \[D_j := X_{j} -X_{j-1} - \E[X_{j} -X_{j-1} \mid \mathcal{F}_{j-1}] = X_j - \E[X_{j} \mid \mathcal{F}_{j-1}],\]
    so that
    \[X_n = X_0 +\sum_{j=1}^n\E[X_{j} -X_{j-1} \mid \mathcal{F}_{j-1}]+  \sum_{j=1}^n D_j \geq X_0 + h n + \sum_{j=1}^n D_j.\]
    We can thus let
    \[R := -\inf_{n\geq0} \big(hn/2+ \sum_{j=1}^n D_j\big),\]
    so
    \[\E e^{qR} \leq \E e^{-q\inf_n (hn/2 + \sum_{j=1}^n D_j)}.\]
    Using that
    \[1+ x \leq e^x \leq 1+ x + x^2 e^{|x|}\]
    and $\E[D_{j+1} \mid \mathcal{F}_j] =0$, we have that for $q \in (0,p/2]$,
    \[\E [e^{- q D_{j+1}} \mid \mathcal{F}_j] \leq 1 -q \E [D_{j+1} \mid \mathcal{F}_j] + q^2 \E [D^2_{j+1} e^{q |D_{j+1}|} \mid \mathcal{F}_j] \leq 1 + C(p,K) q^2 \leq e^{ C(p,K) q^2}.\]
    Thus, by repeatedly conditioning,
    \[\E e^{-q \sum_{j=1}^n D_j} \leq e^{C q^2 n}\]
    Thus choosing $q(p,K,h)$ such that $C q^2 \leq \frac{h}{4} q$, we have that
    \[\E e^{- q ( hn/2+\sum_{j=1}^n D_j)} \leq e^{-qhn/4}.\]
    Thus
    \[ \E e^{-q\inf_n (hn/2 + \sum_{j=1}^n D_j)} = \E \sup_n e^{- q ( hn/2+\sum_{j=1}^n D_j)} \leq \sum_n \E e^{- q ( hn/2+\sum_{j=1}^n D_j)}  \leq C,\]
    allowing us to conclude.
\end{proof}

\begin{proof}[Proof of Theorem~\ref{thm:main}]
    We fix $\kappa \in \big[0,\frac{1}{4\pi^2} \log\frac{\pi}{e}\big]$,  so that $|A^{\kappa,-1}| \geq e$.

    For $n \in \N$, let
    \[Y_n := \begin{cases}  \tfrac{1}{\sqrt{2}}(\mathrm{e}_x - \mathrm{e}_y) \cdot \hat B_n^\kappa(1,0,0)  & n \mathrm{\,mod\,} 3 =0\\
  \tfrac{1}{\sqrt{2}}(\mathrm{e}_y - \mathrm{e}_z) \cdot \hat B_n^\kappa(0,1,0)& n \mathrm{\,mod\,} 3 =1\\
    \tfrac{1}{\sqrt{2}}(\mathrm{e}_z - \mathrm{e}_x) \cdot \hat B_n^\kappa(0,0,1)& n \mathrm{\,mod\,} 3 =2.
    \end{cases}\]
    By Proposition~\ref{prop:one pulse} (with coordinates appropriately cyclically permuted for $n \mathrm{\,mod\,} 3 =1,2$), we have that
    \[Y_{n+1} = A^{\kappa,-1} Y_n e^{-2\pi i \varphi_n} + \tilde{Y}_{n,1} + \tilde{Y}_{n,2} e^{2\pi i \varphi_n},\]
    where $\tilde{Y}_{n,1}, \tilde{Y}_{n,2}$ are $\C$-valued random variables that are measurable with respect to $(\theta_j)_{0 \leq j < n}$ and $\varphi_n$ are uniform $\T^1$-valued random variables given by
    \begin{equation*}
        \varphi_n := \begin{cases}
            \theta_{n,x} - \theta_{n,y} & n \mathrm{\,mod\,} 3 = 0,\\
            \theta_{n,y} - \theta_{n,z} & n \mathrm{\,mod\,} 3 = 1,\\
            \theta_{n,z} - \theta_{n,x} & n \mathrm{\,mod\,} 3 = 2,
        \end{cases}
    \end{equation*}
    hence $\varphi_n$ is $\theta_n$ measurable.

    Letting $X_n := \log |Y_n|$, \eqref{eq:mean growth} then gives that
    \[\E[X_{n+1} \mid (\theta_j)_{0 \leq j < n}]  \geq \log |A^{\kappa,-1}| + X_n \geq 1 + X_n \]
    and \eqref{eq:fluctuation bound} gives that
    \[\E\bigg[\exp\Big(\tfrac{1}{4} \big|X_{n+1} - \E[X_{n+1} \mid (\theta_j)_{0 \leq j < n}]\big|\Big) \,\bigg|\, (\theta_j)_{0 \leq j < n} \bigg] \leq C.\]

    Lemma~\ref{lem:martingale term subcritical} then gives that there exists a random variable $R \geq 0$ and constants  $p,C >0$ (independent of $\kappa \in \big[0,\frac{1}{4\pi^2} \log\frac{\pi}{e}\big]$) such that $\E e^{pR} \leq C$ and, using that $X_0 =0$, for all $n \in \N$,
    \begin{equation}
        \label{eq:X bound}
        X_n \geq \tfrac{1}{2}n - R.
    \end{equation}
    Note that from the definition of $Y_n$, we have that $\|B^\kappa_n\|_{L^2_x} \geq |Y_n|.$ Thus, by the definition of the $X_n$, the Plancherel identity, and~\eqref{eq:X bound}, for all $n \in \N$
    \[\|B^\kappa_n\|_{L^2_x} \geq e^{-R} e^{n/2}.\]
    Using the almost sure bounds on $U$, a simple Gr\"onwall argument gives that for some universal $C>0$, for all $n \geq 1$ and $t \in [n-1,n]$, 
    \[\|B^\kappa_n\|_{L^2_x} \leq C \|B_t^\kappa\|_{L^2_x},\]
    thus we have that for all $t \geq 0$,
    \[\|B^\kappa_t\|_{L^2_x}  \geq C^{-1} e^{-R} e^{t/2}.\]
    Defining $L^\kappa := C e^{R}$ and using the moment bound on $R$, we conclude.
\end{proof}

\section{Discussion of AI use}

\label{s:AI}

The original proof idea for this manuscript was generated essentially autonomously by ChatGPT 5.6 Sol Ultra, using the prompt presented at the end of the section, which was directly adapted from the prompt used in the preprint~\cite{ouimet_proof_2026}. After a short further conversation asking about the proof, I requested a manuscript write-up with the following prompt:
\begin{Verbatim}[breaklines=true, breakanywhere=true]
    Can you write up this result in the form of a research paper in a tex document. Fully prove every step and provide appropriate references.
\end{Verbatim}

The original generated manuscript is available as \verb|ChatGPT_manuscript_original.tex| in the arXiv source files. The original proof contained some minor errors and the headline result was different: everything was phrased in a random dynamical systems setting and bounds were given on the top Lyapunov exponent in the multiplicative ergodic theorem sense. Additionally, the proof was needlessly confusing as it went through a backward-in-time induction argument. The backward-in-time argument was needed because the proof was based essentially on an adjoint version of Proposition~\ref{prop:one pulse}---showing that one Fourier mode is mapped to only three Fourier modes, a column-wise tridiagonal statement---which is less natural to use. Following substantial additional prompting and revision rounds, where I suggested the simplified, forward-in-time argument and the possibility of showing an all-time lower bound statement, eventually a somewhat improved manuscript was generated---available as \verb|ChatGPT_manuscript_final.tex| in the arXiv source files. As part of these revisions, ChatGPT wrote a version of the exponential martingale argument that became Lemma~\ref{lem:martingale term subcritical} in order to prove the all-time lower bound, following a prompt asking if such an inequality is possible.

Ultimately, I determined that the final generated manuscript was rather inadequate in its exposition---making a simple idea difficult to understand---and that instead of attempting to verify the result, I would rather rewrite the argument from scratch. Following a discussion of the result with William Cooperman, which required me to rederive the outline from my understanding of the proof, I became convinced a quite simple argument was available. This manuscript is the write-up of that simple argument. There are a few further simplifications made over the AI-generated manuscript, such as a simpler set of velocity fields that are directly related by coordinate permutation---reducing the computational burden---and a more direct and explicit iteration argument. ChatGPT 5.6 Sol was used in writing of this manuscript: both in explaining some sections of the original AI-generated argument for me to digest as well as copy-editing and mathematical checking. The manuscript has been carefully hand-checked as well.

The main innovation of this manuscript is essentially a case of example-finding, which AI has recently been quite successful at~\cite{alon2026unitdistance,openai2026tenadvances,Alpoge2026}. However, it is apparent---from the experience on this project, other personal interactions with frontier models, and other published, AI-generated results---that current frontier models struggle substantially at writing clear, concise, and human-readable proofs. Without substantial human-provided prodding, the models appear to be rather ``lazy'': once a potential proof strategy emerges, no attempt is made to simplify the proof and isolate the necessary ingredients.

\subsection{Initial prompt}

\begin{Verbatim}[fontsize=\small,breaklines=true, breakanywhere=true]
Current task statement PROBLEM Prove the random fast dynamo conjecture: that is for a suitable class of smooth random velocity fields (OU process on finitely many Fourier modes, or alternating shears with randomized phases), the fields are fast dynamos, that is the (almost surely defined) top Lyapunov exponent for the associated resistive induction equation is uniformly lower bounded above 0 for all diffusivities \kappa \in (0,\ep). -- Assume for purposes of this task that a complete affirmative proof exists. A complete solution must prove exactly the above. Partial progress does not count unless it implies exactly the resolution above. Special cases, reductions to another unproved conjecture, computational verification through any fixed graph size, and candidate counterexamples without a complete nonexistence certificate are insufficient. Use multiagent v2 aggressively and dynamically. Do not use a fixed assignment such as ``N agents for strategy X.'' Instead, manage the search using the following heuristics: - Begin with a genuinely diverse portfolio of approaches. Agents should explore substantially different formulations, invariants, reductions, algebraic viewpoints, structural inductions, decompositions, flow formulations, transition systems, embeddings, extremal arguments, and computational sanity checks. - Do not tell most agents the currently favored approach. Preserve independence during early rounds so that agents do not all converge to the same attractive but incomplete reduction. - Maintain an explicit registry of approach families. Group agents by the mathematical idea they are using, not by superficial wording. If many agents converge to one family, redirect some of them toward underexplored formulations. - Do not allow one approach to dominate merely because it gives elegant reductions. A route that ends at a lemma equivalent in strength to the original conjecture is not close to completion unless it supplies a genuinely new proof of that lemma. - When an approach stalls at a theorem-strength missing lemma, mark that route as blocked. Only continue assigning agents to it if someone proposes a materially new mechanism, invariant, or construction. - Keep several incompatible proof routes alive through multiple rounds. Cross-pollinate ideas only after independent agents have developed them far enough to expose their real strengths and gaps. - Use adversarial agents throughout: every candidate proof must be checked - Require agents to return concrete lemmas, constructions, equations, or counterexamples to proposed sublemmas. Reject status reports, vague optimism, and claims that an unproved global compatibility statement is ``routine.'' - The root agent should repeatedly synthesize, challenge, redirect, and launch new rounds. Do not stop after the first wave fails. Produce a complete proof if one survives audit; otherwise report only the strongest rigorously proved derivation and its exact remaining gap. Do not return merely because current approaches fail or agents report theorem-strength gaps. Continue launching new rounds, reopening blocked approaches only when there is a genuinely new mechanism, and searching for fresh formulations. Return only when a complete affirmative proof has been found and survives adversarial audit. Do not return a reduction, partial result, isolated missing lemma, ``best effort'' summary, or explanation of why the problem is difficult. Public search may be used only for ordinary mathematical background or standard named theorems, not to search for a solution to this exact conjecture or benchmark. Do not search the public web merely to determine whether it is open, and do not answer that it is open.
\end{Verbatim}

\subsection*{Acknowledgments}

I would like to thank William Cooperman for working through the argument with me on the chalkboard during a visit to ETH Z\"urich.

{\small
\bibliographystyle{alpha}
\bibliography{cleanreferences}
}

\end{document}